\documentclass{amsart}
\usepackage{amsfonts}
\usepackage{amssymb,amsmath,amsfonts,verbatim}
\usepackage{graphics}
\usepackage{epsfig}
\usepackage{color}
\usepackage{hyperref}
\usepackage{bbm}

\theoremstyle{plain} 
    \newtheorem{theorem}{Theorem}
    \newtheorem{lemma}[theorem]{Lemma}
    \newtheorem{proposition}[theorem]{Proposition}
    \newtheorem{corollary}[theorem]{Corollary}

\theoremstyle{definition} 
    \newtheorem{definition}[theorem]{Definition}

    \newtheorem{remark}[theorem]{Remark}

\theoremstyle{remark}

\def\l{\left}
\def\r{\right}
\def\<{\langle}
\def\>{\rangle}

\def\mb{\mbox}
\newcommand{\E}{\mbox{\bf E}}

\def\bar{\overline}
\def\P{{\bf P}}

\def\given{\left.\vphantom{\hbox{\Large (}}\right|}

\newcommand\mnote[1]{} 
\newcommand\be{\begin{equation*}}
\newcommand\ee{\end{equation*}}

\newcommand\ben{\begin{equation}}
\newcommand\een{\end{equation}}
\newcommand\bes{\begin{eqnarray*}}
\newcommand\ees{\end{eqnarray*}}

\newcommand{\F}{\langle\alpha,\beta\rangle}
\newcommand{\FF}{\langle\alpha_1,\dots,\alpha_k\rangle}
\newcommand{\N}{\mathbb{N}}

\def\mb{\mbox}
\def\\ell{\left}
\def\r{\right}

\def\lam{\lambda}
\def\alp{\alpha}

\usepackage{palatino}

\author{Siddhartha Gadgil}
\email{gadgil@iisc.ac.in}

\author{Manjunath Krishnapur}
\email{manju@iisc.ac.in}

\address{	Department of Mathematics\\
		Indian Institute of Science\\
		Bangalore 560012, India}
		
\thanks{The first author is partially supported by the SERB-EMR grant EMR/2016/006049. The second author is partially supported by the SERB-MATRICS grant MTR2017/000292.   Both authors are partially supported by the UGC centre for advanced studies.}		

\title[Random words and non-crossing matchings]{Random words in free groups, non-crossing matchings and  RNA secondary structures}

\begin{document}
\bibliographystyle{abbrv}
\begin{abstract}
  Consider a random word $X^n=(X_1,\ldots ,X_n)$ in an alphabet consisting of $4$ letters, with the letters viewed either as $A$, $U$, $G$ and $C$ (i.e., nucleotides in an RNA sequence) or $\alpha$, $\bar{\alpha}$, $\beta$ and $\bar{\beta}$ (i.e., generators of the free group  $\F$ and their inverses).  We show that the expected fraction $\rho(n)$ of unpaired bases in an optimal RNA secondary structure (with only Watson-Crick bonds and no pseudo-knots) converges to a constant $\lambda_2$ with $0<\lambda_2<1$ as $n\to\infty$. Thus, a positive proportion of the bases of a random RNA string do not form hydrogen bonds. We do not know the exact value of $\lambda_2$, but we derive upper and lower bounds for it.

  In terms of free groups, $\rho(n)$ is the ratio of the length of the shortest word representing $X$ in the generating set consisting of conjugates of generators and their inverses to the word length of $X$ with respect to the standard generators and their inverses. Thus for a typical word the word length in the (infinite) generating set consisting of the conjugates of standard generators grows linearly with the word length in the standard generators. In fact, we show that a similar result holds for all non-abelian finitely generated free groups $\FF$, $k\geq 2$.

\end{abstract}
\maketitle

\section{Introduction}

Consider a word $X$ in an alphabet consisting of $4$ letters, with the letters viewed either as $\alpha$, $\bar{\alpha}$, $\beta$ and $\bar{\beta}$ (i.e., generators of the free group  $\F$ and their inverses, where we use the notation $\bar{g}$ for $g^{-1}$) or $A$, $U$, $G$ and $C$ (i.e., nucleotides in an RNA sequence).  There is a natural notion of a length $\ell(X)$ associated to such a word, which can be defined in several equivalent ways (see~\cite{Ga1} and~\cite{Ga2} for more details). We give three descriptions of $\ell$, two of which (as we indicate below) generalize to random words in $2k$ letters, for $k\geq 2$.

\begin{enumerate}
  \item If $X$ is viewed as a word in $\F$ then $\ell$ is the maximal conjugacy-invariant length function on $\F$ which satisfies $\ell(\alpha)\leq 1$ and $\ell(\beta)\leq 1$. Equivalently, $\ell$ is the word length in the generating set given by all conjugates $g\alpha g^{-1}$, $g\bar{\alpha} g^{-1}$, $g\bar{\beta} g^{-1}$ and $g\beta g^{-1}$ of the generators of $\F$ and their inverses (where $\bar{\alpha}=\alpha^{-1}$ and $\bar{\beta} = \beta^{-1}$). More generally, an arbitrary word in $2k$ letters gives an element of $\FF$, and $\ell$ can be defined as a maximal conjugacy-invariant length function (or word length in conjugates of generators and their inverses) in this case too.
  \item If $X$ is viewed as a nucleotide sequence, then we can consider so called \emph{secondary structures} of RNA~\cite{RNA}, i.e., bonds between nucleotides of the RNA, with bonds being Watson-Crick pairs, i.e. hydrogen bonds between Adenine and Uracil and between Guanine and Cytosine, and stereo-chemical forces modelled by not allowing so called \emph{pseudo-knots} (for details we refer to~\cite{Ga1}). Then $\ell(X)$ is the minimum number of non-bonded nucleotides for secondary structures of $X$. This is a biologically reasonable notion of energy.
  \item Again viewing $X=X^{(n)}$ as a word of length $n$ in the alphabet $\alpha$, $\bar{\alpha}$, $\beta$ and $\bar{\beta}$, we consider incomplete non-crossing matchings of the (indices of) letters in $X$ so that letters are matched with their inverses. Here a non-crossing matching is a set $P$ of pairs of indices $(i, j)$, $1\leq i < j\leq n$, such that
        \begin{enumerate}
          \item  each $i$ belongs to at most one element of $P$,
          \item  if $i<j<k<\ell$, then at most one of $(i, k)$ and $(j,\ell)$ belong to $P$,
          \item  if $(i, j)\in P$ then $X_i = \bar{X}_j$.
        \end{enumerate}
        The length $\ell(X)$ is the minimum number of unmatched letters over all non-crossing matchings. More generally we can take a random word in the alphabet with $2k$ letters $\alpha_1$, $\bar{\alpha}_1$, \dots, $\alpha_k$, $\bar{\alpha}_k$ (where $\bar{g}$ denotes $g^{-1})$ and consider non-crossing matchings with letters paired with their inverses, and define $\ell$ as the minimum number of unmatched letters over all non-crossing matchings.
\end{enumerate}

Henceforth, fix $k \geq 2$ and consider a \emph{random} string $X=X^{(n)}$ of length $n$ in $2k$ letters as above (i.e., a random word). The case $k = 2$ corresponds to RNA secondary structures, but most of our results and proofs are uniform in $k$. Let $L_{k}(n)=\E\l[\ell(X^{(n)})\r]$ where the expectation is over uniform distribution on strings of length $n$. Let $\rho_{k}(n)=L_{k}(n)/n$ denote the average proportion of unpaired letters.

Our main result is that this fraction converges to a positive constant.
\begin{theorem}\label{main} With the above notations, $\rho_{k}(n)\rightarrow \lam_{k}$ for some constant $0<\lambda_k<1$.
\end{theorem}
Thus, the average proportion of unpaired bases in an optimal secondary structure for a random RNA string converges to a \emph{positive} constant as the length of the RNA string approaches infinity. Equivalently, for a word $X$ in the free group $\F$ (or more generally in the free group $\FF$ for $k\geq 2$), the average ratio of the word length of $X$ in the (infinite) generating set consisting of conjugates of generators and their inverses to the word length of $X$ in the standard generators and their inverses converges to a positive constant. We remark that this result is also true, but essentially trivial, for the free group $\mathbb{Z}$ on $1$ generator (for the group $\mathbb{Z}$, the two generating sets, hence the corresponding word lengths, coincide).

We also show  that  $\ell(X^n)/n$ has exponential concentration in a window of length $1/\sqrt{n}$ around its expectation $\rho_k(n)$, and hence around $\lambda_{k}$.
\begin{proposition}\label{prop:concentration}
  $\P\l\{\given \ell(X^{(n)}) - n\rho_k(n) \given > t\sqrt{n} \r\} \le 2e^{-\frac{t^{2}}{8}}$ for any $t>0$.
\end{proposition}
An immediate corollary is that the standard deviation of $\ell(X^n)$ is $O(\sqrt{n})$.

As for proofs, the existence of the limit $\lambda_k$ and the exponential concentration are proved using sub-additivity and Hoeffding's inequality respectively, which are standard methods in combinatorial optimization problems. Showing that $\lambda_k$ is strictly positive, and getting  bounds for its value require more involved arguments. It would be  interesting to find the exact value of $\lambda_k$, particularly $\lambda_2$. We are only able to get  bounds.

For $k=2$, we prove the explicit bounds $0.034< \lambda_2< 0.231$. The proof of Theorem~\ref{main} given in Section~\ref{S:Positive} gives the lower bound of $0.03$, which is then refined to get the slightly better lower bound of $0.034$. Elementary arguments in Section~\ref*{S:upperbound} give an upper bound of $0.289$ which is improved to $\frac{3}{13}=0.2307\ldots$  in Section~\ref{Sec:greedy}. This is achieved by analysing a specific algorithm for producing a non-crossing matching described below.

\subsection*{The one-sided greedy algorithm}   Scan the letters $X_{1},X_{2},\ldots$  in that order and when the turn of $X_{t}$ comes (starting from $t=1$), match it to  $X_{s}$ with the largest value of $s<t$, if possible (i.e., $X_{s}=\overline{X}_{t}$, and there is no $u\in (s,t)$ such that $X_{u}=\overline{X}_t$, and the non-crossing condition is maintained).

For example, if $k=2$ and the word is $\alpha \beta\alpha\beta\bar{\alpha}\alpha\bar{\beta}\beta$, then the matching is $3\leftrightarrow 5$, $2\leftrightarrow 7$ (here $3,5,2,7$ represent the indices in the word, of course).

\begin{proposition}\label{prop:onesidedgreedy} In the one-sided greedy algorithm, the proportion of unmatched letters converges to
  \begin{align}\label{eq:boundforlambdakfromgreedy}
    \tilde{\lambda}_k= 1-\frac{\sum_{r=1}^kr2^r\binom{k}{r}\prod_{j=1}^r\frac{j(j+1)}{j(2k-j)-1}}{k\sum_{r=0}^k\binom{k}{r}2^r\prod_{j=1}^r\frac{j(j+1)}{j(2k-j)-1}}.
  \end{align}
  Therefore $\lambda_k\le \tilde{\lambda}_k$.
\end{proposition}
The numerical values of upper bound for the first few $k$ are

\begin{tabular}{c|c|c|c|c}
  $k$                 & $2$                        & $3$                   & $4$                            & $5$                             \\ \hline
  $\tilde{\lambda}_k$ & $\frac{3}{13}=0.231\ldots$ & $\frac{33}{100}=0.33$ & $\frac{297}{455}=0.393\ldots $ & $\frac{3126}{7115}=0.439\ldots$
\end{tabular}

\medskip
Proposition~\ref{prop:onesidedgreedy} is proved by analysing an associated  Markov chain on the space of words. This Markov chain is described in Section~\ref{Sec:greedy}, where we also find its stationary distribution explicitly. It may be of  independent interest, as there are  not many examples of chains that are neither reversible nor have a doubly stochastic transition matrix for which we can solve for the stationary distribution exactly.

There is some slack in our proofs, so our bounds can be sharpened. However our goal here is to give a simple and transparent proof.  In fact certain enumerative algorithms suggest that $\lambda_2<0.11$ but we are unable to analyse these algorithms rigorously.

\subsection*{Dependence of $\lambda_k$ on $k$} One may also ask about the behaviour of $\lambda_k$ as a function of $k$. We claim that $\lambda_k\le \lambda_{k+1}$. This is easiest seen by coupling. Consider a random word $X^n$ using symbols $\alpha_i, \bar{\alpha}_i$, $1\le i\le k+1$. Let $X_{(j)}$ denote the word got by deleting all occurrences of $\alpha_j,\bar{\alpha}_j$ in $X^{(n)}$, and let $N_j$ be the length of $X_{(j)}$. Let $\ell_{(j)}$ denote the number of unmatched letters when  the optimal matching on $X^n$ is restricted to $X_{(j)}$. Then $\ell_{(1)}+\ldots +\ell_{(k+1)}=k\ell(X^n)$ and hence taking expectations and using symmetry, 
\begin{align}\label{eq:inequalityofLkLkplus1}
(k+1)\E[L_k(N_{1})]\le k L_{k+1}(n).
\end{align}
The expectation on the left is over the randomness in $N_{1}$ which has Binomial distribution with parameters $(n,k/(k+1))$. By Chebyshev's inequality, $\P\{n_-\le N_1\le n_+\}\ge 1-O(n^{-\frac12})$, where $n_{\pm}=\frac{kn}{k+1}\pm n^{\frac34}$.  As $n\mapsto L_k(n)$ is  obviously  increasing in $n$, 
\[
\E[L_k(N_1)]\ge (1-O(n^{-\frac12}))L_k(n_-).
\]
Combine this with \eqref{eq:inequalityofLkLkplus1}, divide by $n$, and let $n\to \infty$ to get  $\lambda_k\le \lambda_{k+1}$.

Further, we show in Proposition~\ref{prop:k-to-1} that $\lambda_k\to 1$ as $k\to \infty$.

\begin{remark}
  As a consequence of the convergence of the fraction unmatched to a positive constant and the concentration result, it follows that there is some \emph{scale} $N$ so that, for a generic RNA strand, optimal structures on pieces of length $N$ can be concatenated to give a near-optimal structure on the whole strand. As bonds at long distances are less likely to form, it follows that RNA folding can be localized to this scale, which makes foldings easier to analyse.
\end{remark}

\subsection*{Outline of the paper}  In Section~\ref{sec:preliminaries} we show that the different ways of defining the length $\ell$ outlined above give the same function. In Section~\ref{S:Positive} we prove Theorem~\ref{main} and the above-stated lower bounds for $\lambda_2$. In Section~\ref{S:upperbound} we present an elementary argument to obtain the upper bound of $0.289$ for $\lambda_2$. In Section~\ref{sec:concentration}, we prove Proposition~\ref{prop:concentration}. In Section~\ref{Sec:greedy} we introduce the Markov chain associated to the one-sided greedy algorithm, and  explicitly analyse it prove Proposition~\ref{prop:onesidedgreedy}. In particular, this leads to the improved upper bound for $\lambda_2$.

\section{Preliminaries}\label{sec:preliminaries}

For the convenience of the reader, we define length functions on groups and show that three definitions of the length $\ell$ on $\FF$ given above give the same function. The results in this section are elementary.

\begin{definition}
  Let $G = (G, \cdot, e, (\cdot)^{-1})$ be a group (written multiplicatively,
  with identity element $e$).  A \emph{length function} on $G$ is a
  map $l : G \to [0,+\infty)$ that obeys the properties
  \begin{itemize}
    \item $l(e) = 0$,
    \item $l(x)>0$, for all $x \in G\setminus \{e\}$,
    \item $l(x^{-1}) = l(x)$, for all $x,y\in G$.
    \item $l (x y) \leq l(x) + l(y)$, for all $x,y\in G$.
  \end{itemize}
\end{definition}

\begin{definition}
  We say that a length function $l$ is \emph{conjugacy-invariant} if $l(xyx^{-1}) = l(y)$ for all $x,y\in G$.
\end{definition}

We shall see here that three definitions of a length $\ell:\FF \to [0, \infty)$ coincide. We also give more details of these definitions.

\subsection{Maximal length}
Consider the set $\mathcal{L}$ consisting of conjugacy-invariant length functions $l: \FF\to [0, +\infty)$ satisfying $l(\alpha_i)\leq 1$ for all $1\leq i\leq k$. We have a partial order on length functions on $\FF$ given by $l_1 \leq l_2$ if and only if $l_1(g)\leq l_2(g)$ for all $g\in\FF$. For this order, it is well known that there is a (necessarily unique, by properties of posets) maximal element. Namely, define
$$\ell_{max}(g) = \sup\{l(g): l\in\mathcal{L}\}.$$

Note that the set $\{l(g): l\in\mathcal{L}\}$ is bounded by the word length of $g$, so has a supremum.
It is easy to see that $\ell_{max}$ is a conjugacy-invariant length function, and that $\ell(\alpha_i)\leq 1$ for all $1\leq i\leq k$. Thus $\ell_{max}\in \mathcal{L}$. Further, by construction, if $l\in\mathcal{L}$, then $l\leq \ell_{max}$. Thus $\ell_{max}$ is the maximum of the set $\mathcal{L}$.

\subsection{Word length in conjugates of generators}

Let $\ell_{CW}: \FF \to [0, +\infty)$ be the function given by the  word length in the generating set consisting of all conjugates of the generators $\alpha_i$, $1\leq i \leq k$. Thus, for $g\in\FF$, $\ell_{CW}(g)$ is the smallest value $r\geq 0$ so that $g$ can be expressed as
$$g = \prod_{j=1}^r \beta_j\alpha_{i_j}^{\epsilon_j}\beta_j^{-1},$$
where $\beta_j\in \FF$ and $\epsilon_j=\pm 1$, for $1 \leq j\leq r$.

\begin{proposition}
  We have $\ell_{CW} = \ell_{max}$.
\end{proposition}
\begin{proof}
  We see that $\ell_{CW}\in\mathcal{L}$. This is because the word length in a conjugacy-invariant set is a conjugacy-invariant length function, and $\ell_{CW}(\alpha_i) = 1$ for $1\leq i\leq k$.

  Further, we see that $\ell_{CW}$ is maximal. Namely, let $l\in\mathcal{L}$, $g\in G$ and let $r = l_{CW}(g)$. Then we can express $g$ as $g = \prod_{j=1}^r \beta_j\alpha_{i_j}^{\epsilon_j}\beta_j^{-1}$. By the triangle inequality, conjugacy-invariance, symmetry, and using $l(\alpha_i)\leq 1$ for $i\leq i\leq k$,
  $$l(g)\leq\sum_{j=1}^r l(\beta_j\alpha_{i_j}^{\epsilon_i}\beta_j^{-1})\leq\sum_{j=1}^r l(\alpha_{i_j}^{\epsilon_i})\leq \sum_{j=1}^r 1 = r = \ell_{CW}(g),$$
  as required

  As $\ell_{CW}\in \mathcal{L}$ is maximal, $\ell_{CW}=\ell_{max}$.
\end{proof}

\subsection{Length from non-crossing matchings}

Let $X^{(n)}=(X_{1},\ldots ,X_{n})$ be a word in the alphabet with $2k$ letters $\alpha_1$, $\bar{\alpha}_1$, \dots, $\alpha_k$, $\bar{\alpha}_k$. Let $NC$ stand for incomplete non-crossing matchings of $[n]= \{1, 2,\dots,n\}$. Let $NC_{k}(X)$ be the subset of $M\in NC$ such that for each matched pair $(i,j)\in M$ we have $\{X_{i},X_{j}\}=\{\alp_{\ell},\bar{\alp}_{\ell}\}$ for some $\ell \le k$.

let $\ell_{NC}(X)$ be the minimum number of unmatched pairs in all non-crossing matchings so that letters are paired with their inverses. We sketch the proofs that this is well-defined on $\FF$, a conjugacy invariant length function and that $\ell_{NC} = \ell_{max}$. For more details, see~\cite{Ga2} (which however has different terminology, and considers proofs for the case of two generators, though the proofs work just the same for general $k$).

\begin{lemma}\label{welldefined}
  Suppose $X_1$ and $X_2$ represent the same element in the group $\FF$, then $\ell_{NC}(X_1)= \ell_{NC}(X_2)$.
\end{lemma}
\begin{proof}
  It suffices to consider the case where $X_1$ and $X_2$ are related by a single cancellation. Without loss of generality, assume that there exist words $W_1$ and $W_2$ and an index $1\leq j\leq k$ such that $X_1= W_1W_2$ and $X_2=W_1\alpha_j\alpha_j^{-1}W_2$. Let $\mu_p$ be the length of $W_p$ for $p = 1,2$. Note that the cancelling pair corresponds to the pair $(\mu_1+1, \mu_1+ 2)$ of indices.

  We show that $\ell_{NC}(X_1)= \ell_{NC}(X_2)$. First, fix a non-crossing matching $M_1$ of $X_1$ with $\ell_{NC}(X_1)$ unmatched letters and with letters paired with their inverses. Let $\sigma:\N\to\N$ be defined by
  $$\sigma(m) = \begin{cases}
      m     & \textrm{if $m\leq \mu_1$}, \\
      m + 2 & \textrm{if $m > \mu_1$}
    \end{cases}$$
  Then $M_2 := \sigma(M_1)\cup\{(\mu_1+1, \mu_1 + 2)\}\in NC_k(X_2)$ and has $\ell_{NC}(X_1)$ unmatched letters (i.e., the same as $M_1$). Hence $\ell_{NC}(X_2)\leq \ell_{NC}(X_1)$.

  Conversely, fix a non-crossing matching $M_2\in NC_k(X_2)$ with $\ell_{NC}(X_2)$ unmatched letters. Suppose at most one of $\mu + 1$ and $\mu + 2$ is matched in $M'$, and $(i, j)\in M$ is the corresponding pair with $j\in \{\mu +1, \mu+ 2\}$. Then $M_1 :=M_2\setminus\{(i, j)\}\in NC_k(X_1)$ and $M_1$ has at most $\ell_{NC}(X_2)$ unmatched letters.

  Next, if $(\mu + 1, \mu + 2)\in M_2$, then $M_1:= M_2\setminus\{(\mu + 1, \mu + 2)\}\in NC_k(X_1)$ and $M_1$ has  $\ell_{NC}(X_2)$ unmatched letters.

  Finally, if for some indices $i$ and $j$ we have $(i, \mu + 1)\in M_2$ and $(j, \mu + 2)\in M_2$ (after possibly flipping some pairs), we can see that $$M_1:= M_2\cup\{(\mu + 1, \mu + 2)\}\setminus{\{(i, \mu + 1), (j, \mu + 2)\}}\in NC_k(X_1)$$ and $M_1$ has $\ell_{NC}(X_2)$ unmatched letters.

  In all cases, we conclude that $\ell_{NC}(X_1)\leq \ell_{NC}(X_2)$.

\end{proof}

It follows that $\ell_{NC}$ induces a well-defined function on $\FF$, which we also denote as $\ell_{NC}$. It is easy to see that it is a length function. The proof of the following is very similar to that of Lemma~\ref{welldefined}.

\begin{lemma}\label{conjinv}
  Suppose $g, h\in\FF$, then $\ell_{NC}(hgh^{-1})= \ell_{NC}(g)$.
\end{lemma}
\qed

It is easy to see that $\ell_{NC}(\alpha_i) = 1$ for all $1\leq i\leq k$, and that $\ell_{NC}$ is symmetric. Thus $\ell_{NC}\in\mathcal{L}$. Hence, to show that $\ell_{NC} = \ell_{max}$ it suffices to prove maximality, which we prove next.

\begin{lemma}\label{maximal}
  Suppose $l\in\mathcal{L}$ and $g\in \FF$. Then $l(g)\leq\ell_{NC}(g)$.
\end{lemma}
\begin{proof}
  Let $X$ be a word representing $g$. We prove the lemma by (strong) induction on the length $n$ of $X$. The case when the length is zero is clear.
  Consider a non-crossing matching $M\in NC_k(X)$ with $\ell_{NC}(X)$ unmatched letters. First, suppose the index $1$ is unmatched in $M$, let $\widehat{X}$ be obtained from $X$ by deleting the first letter. Then $M\in NC_k(\widehat{X})$, so by induction hypothesis, $l(\widehat{X})\leq \ell_{NC}(\widehat{X})$. Further, as $M$ restricted to $\widehat{X}$ has one less unmatched letter than $M$, we conclude that $\ell_{NC}(X) = \ell_{NC}(\widehat{X}) + 1$. As the first letter of $X$ is a generator or the inverse of a generator, using the triangle inequality
  $$l(X)\leq 1 + l(\widehat{X})\leq 1 + \ell_{NC}(\widehat{X})=\ell_{NC}(X).$$

  Next, if the pair $(1, j)\in M$ with $j <n$, we split the word $X$ as $X = X_1 * X_2$ with $X_1$ of length $j$. Observe that the non-crossing condition implies that $M$ decomposes as $M_1\cup M_2$ with $M_1\in NC_k(X_1)$ and $M_2\in NC_k(X_2)$. Again, we use the induction hypothesis and the triangle inequality to conclude that $l(X)\leq \ell_{NC}(X)$.

  Finally, if $(1, n)\in M$, let $\widehat{X}$ be obtained from $X$ by deleting the first and last letter. By conjugacy invariance of $l$ and $\ell_{NC}$, $l(X) = l(\widehat{X})$ and $\ell_{NC}(X)= \ell_{NC}(\widehat{X})$. Applying the induction hypothesis to $X$ gives the claim.
\end{proof}

Thus, we can conclude the following.

\begin{proposition}
  We have $\ell_{NC} = \ell_{max}$.
\end{proposition}

\section{The proportion of unmatched indices}\label{S:Positive}
In this section, we prove Theorem~\ref{main} and get lower bounds on $\lambda_{2}$. At first, $k$ is fixed, hence we drop it in the subscripts of $L(n)$ and $\rho(n)$.

The first observation is  that $L(n)$ is sub-additive.

\begin{lemma}\label{lem:subadditivityofL}
  For $m,n>0$, $L(m+n)\leq L(m)+L(n)$.
\end{lemma}
\begin{proof}
  A string $X^{(m+n)}$ of i.i.d. random variables of length $m+n$  is obtained by taking the concatenation $X_1^{(n)}*X_2^{(m)}$ of two strings $X_1^{(n)}$ and $X_2^{(m)}$ of i.i.d. random variables of lengths $n$ and $m$ respectively. As the union of elements  $M_1=NC_k(X_1)$ and $M_2=NC_k(X_2)$ gives a matching $M\in NC_k(X)$, it is easy to see that $\ell(X)\leq \ell(X_1)+\ell(X_2)$. By taking expectations the lemma follows.
\end{proof}

As a well known consequence of sub-additivity (Fekete's lemma), we obtain the following.

\begin{corollary}
  The sequence $\rho(n)=L(n)/n$ converges to  $\lambda_k:=\inf_{n} \rho (n)$.
\end{corollary}

As $0\leq \rho(n)\leq 1$, we get $0\leq \lambda_k\leq 1$. It is easy to get some upper bounds for $\lambda_k$ by computing $\rho(n)$ for small $n$ (as $\lambda_k$ is the infimum of $\rho(n)$). For instance, for $k = 2$ and $n= 4$, $\ell(X)$ takes values $0$, $2$ and $4$ with probabilities $28/256$, $168/256$ and $60/256$, respectively, hence $\lambda_2\leq \rho(4) = 9/16$. The harder thing is to get lower bounds. Our main result is that $\lambda_k$, which is the asymptotic proportion of unpaired bases, is positive.

\begin{lemma}\label{lem:positivityofc0}
  We have $\lambda_k>0$.
\end{lemma}
\begin{proof}

  Fix $\delta >0$. Observe that $$L(n) \geq n\delta\cdot \P(\ell(X^{(n)}) \geq n\delta),$$ and hence $$\rho(n) \geq \delta \P(\ell(X^{(n)}) \geq n\delta).$$ Thus, if we have $\P(\ell(X^{(n)}) \geq n\delta) \to 1$ as $n\to\infty$, then $\lambda_k\geq \delta$. Thus it suffices to find a $\delta>0$ for which we can show that $\P(\ell(X^{(n)}) \geq n\delta) \to 1$,  or equivalently show that $$\P(\ell(X^{(n)}) <  n\delta) \to 0$$ as $n\to\infty$.

  We shall now bound $\P(\ell(X^{(n)}) <  n\delta)$ for small enough $\delta$. Note that if $W(n, \delta)$ is the number of words $X$ of length $n$ with $\ell(X)< n\delta$, then $$\P(\ell(X^{(n)}) <  n\delta) = \frac{W(n, \delta)}{(2k)^n}.$$

  Let $m=\left\lceil\frac{n - n\delta}{2}\right\rceil$ and let $r = n - 2m$. Observe that if $\ell(X^{(n)})< n\delta$, then $X= X^{(n)}$ has a non-crossing matching with at least $2m$ pairs, and hence a non-crossing matching $M$ with exactly $2m$ pairs (by simply dropping a few pairs). Given such an $M$, we can associate to $X$ a triple $(Y, Z, s)$ where
  \begin{itemize}
    \item $Y$ is the word (of length $r$) consisting of the letters of $X$ that are  \emph{unmatched} in $M$, in the same order as in $X$,
    \item $Z$ is the word (of length $2m$) consisting of the letters of $X$ that are  \emph{matched} in $M$, in the same order as in $X$, and,
    \item $s$ is the set of indices $i$, $1 \leq i\leq n$, that are \emph{unmatched}.
  \end{itemize}

  Note that $M$ gives a \emph{complete non-crossing matching} on $Z$, and hence $Z$ represents the trivial word in $\FF$. As the triple $(Y, Z, s)$ determines $X$, it follows that the number $W(n, \delta)$ of words $X$ of length $n$ with $\ell(X)< n\delta$ is bounded above by the number of triples $(Y, Z, s)$, with
  \begin{itemize}
    \item $Y$ a word of length $r$,
    \item $Z$ a word of length $2m$ that represents the trivial element in the free group, and
    \item $s$ a subset of size $r$ of $\{1, 2, \dots,n \}$.
  \end{itemize}

  Let $T_p$ denote the set of words of length $p$ that represent the trivial element in the group $\FF$. It follows that
  \begin{equation}\label{triple-count}
    |W(n, \delta)| \leq \binom{n}{r}\cdot (2k)^r\cdot |T_{2m}|
  \end{equation}

  The main step remaining is to bound $T_{2m}$. Let $\tau_p = T_p/ (2k)^p$ represent the probability that a random word of length $p$ represents the trivial element in $\FF$. We observe that this is the probability that the standard symmetric random walk on the Cayley graph of the free group (with the canonical generators and their inverses) starting at the identity returns to the identity in $p$ steps. It is clear that $\tau_{p}\tau_{q}\le \tau_{p+q}$, and hence by the Fekete lemma (applied to $\log \tau_{p}$), we see that $\tau_{p}^{\frac{1}{p}}\to \theta_{k}:=\sup_{p}\tau_{p}^{\frac{1}{p}}$. This means that $\tau_{p}\le \theta_{k}^{p}$ for each $p\ge 1$. Of course $\tau_{p}=0$ for odd $p$.

  It is a known fact that $\theta_{k}= \sqrt{\frac{2k-1}{k^2}}$ (for example Kesten~\cite{Ke}). To see this, observe that the graph distance of the random walk to the identity element is itself a random walk on $\N=\{0,1,2,\ldots\}$ that goes from $i\mapsto i+1$ with probability $(2k-1)/2k$ and $i\mapsto i-1$ with probability $1/2k$, for $i\ge 1$, and from $0$ to $1$ with probability $1$. The number of walks of length $p=2m$ that return to the origin in $\N$ is the Catalan number $\frac{1}{m+1}\binom{2m}{m}$, and each such path (since it has $m$ up-steps and $m$ down-steps) has probability $(2k-1)^{m}/(2k)^{2m}$. Therefore,
  \begin{align*}
    \tau_{2m} & = \frac{(2k-1)^{m}}{(2k)^{2m}(m+1)}\binom{2m}{m}              \\
              & \sim \frac{1}{\sqrt{\pi}m^{\frac32}}\frac{(2k-1)^{m}}{k^{2m}}
  \end{align*}
  by Stirling's formula, where $a_{m}\sim b_{m}$ means that $a_{m}/b_{m}$ converges to $1$ as $m\to \infty$. In particular, we see that $\tau_{2m}^{\frac{1}{2m}}\to \sqrt{\frac{2k-1}{k^2}}$. Hence $\theta_{k}=\sqrt{\frac{2k-1}{k^2}}$. In particular, $\theta_2 = \frac{\sqrt{3}}{2}$, which we use for explicit estimates on $\lambda_{2}$.

  %

  It is now straightforward to complete the proof. For simplicity of notation, we ignore the error in rounding off to an integer and assume $r = n\delta$. Using  the elementary fact that $\binom{n}{r}\le e^{nh(\delta)}$ where $h(\delta)= -\delta\log(\delta) - (1 - \delta)\log(1-\delta)$ in \eqref{triple-count}, we get
  $$\P(\ell(X^{(n)}) < n\delta)\leq \exp\left\{n\left(h(\delta)+ \log \theta_k ) \right) \right\}$$
  Hence $\P(\ell(X^{(n)}) < n\delta)\to 0$ as $n\to\infty$ provided $h(\delta)+ \log\theta_k < 0$.

  When $k = 2$, as $\theta_{2}=\sqrt{3}/2$, this  happens, for example, for $\delta = 0.03$. Thus, we have $\lambda_2 > 0.03$, i.e. at least 3\% of the letters are unmatched for the best non-crossing matching for most words.

  Next, suppose $k\to\infty$. We see that $\lambda_k\to\infty$.
  \begin{proposition}\label{prop:k-to-1}
    We have $\lim_{k\to\infty} \lambda_k = 1$.
  \end{proposition}
  \begin{proof}
    Observe that $\theta_{k}= \sqrt{\frac{2k-1}{k^2}}\to 0$ as $k\to\infty$, hence $\log(\theta_k)\to -\infty$. It follows that for any fixed $\delta\in (0, 1)$,
    if $k$ is sufficiently large we have $h(\delta)+ \log\theta_k < 0$, hence $\lambda_k > \delta$. As
    $\lambda_k\leq 1$ for all $k$, $\lim_{k\to\infty} \lambda_k = 1$.
  \end{proof}

  Thus, we have shown that the limit $\lambda_k$ of the sequence $\rho(n)$ exists and is positive.
  This completes the proof of Theorem~\ref{main}, with the effective bound $\lambda_2\geq 0.03$ for $k = 2$ (other effective bounds can be computed similarly).
\end{proof}


\subsection{Refinement of the lower bound for \texorpdfstring{$\lambda_{2}$}{L2} using maximal triples}

We can refine the bound we obtained by choosing the triple $(Y, Z, s)$ in a canonical way (note that we do not, however, choose a canonical non-crossing matching on $Z$). Namely we try to match letters with as low indices as possible among all minimal non-crossing matchings. We fix $k = 2$ (so $\FF = \F$) in this subsection.

First, observe that for fixed $X$, the words $Y$ and $Z$ are determined by $s$. More generally, given $X$, any subset $s\subset [n]$ determines words $Y= Y(X, s)$ and $Z= Z(X, s)$, but in general the word $Z(X, s)$ may not represent the trivial element in $\FF$. We shall say the triple $(Y(s), Z(s), s)$ determined by $s$ (and $X$) is \emph{admissible} provided $Z$ represents the trivial element.

The set $s$ can be viewed as a finite sequence by ordering its elements lexicographically, and two such sets can be compared using the lexicographic ordering on finite sequences, which is a total ordering. We order admissible triples $(Y, Z, s)$ by the component $s$ and choose the maximal admissible triple for each fixed $X$.

We can decompose $s$ as $s = s_1\cup s_2$, with $s_2$ (the {tail}) consisting of those elements $i\in s$ such that if $j\in [n]\setminus s$, then $j < i$. Conversely, given an element $i \in s_1$ there exists $j\in [n]\setminus s$ such that $j > i$. For $i\in s_1$, let $\hat{i}$ be the \emph{smallest} element in $[n]\setminus s$ such that $\hat{i} > i$, i.e., $\hat{i}$ is the first matched index after the unmatched index $i$. Geometrically, the unmatched indices $s$ are in general interspersed with the matched indices, with a (possibly empty) tail $s_2$ of unmatched indices which are larger than all matched indices.

We claim that if $(X, Y, s)$ is maximal and $i\in s_1$, then $X_i\neq X_{\hat{i}}$. For, if $X_i = X_{\hat{i}}$, let $s' = s \setminus\{i\}\cup\{\hat{i}\}$, $Y'= Y(X, s')$  and $Z' = Z(X, s')$. Then $Z' = Z$ as words in the free group, as the letter $X_{\hat{i}}$ in $Z$ has been replaced by $X_i = X_{\hat{i}}$ in $Z'$, and in the order on indices, $i$ has the same position in $Z'$ as $\hat{i}$ has in $Z$ (this is because, if $j\in [n]\setminus(s\cup s')$ is an index in both $Z$ and $Z'$, then $j\leq i$ if and only if $j\leq \hat{i}$ by definition of $\hat{i}$). Hence $Z'$ represents the trivial word. Hence the triple $(Y', Z', s')$ is admissible, and $s$ and $s'$ have the same cardinality. But $s < s'$, contradicting maximality of $s$.

Thus, writing $Y = Y_1 * Y_2$ with $Y_i$ the word with letters $X_j$, $j\in s_i$, and letting $r_i$ be the cardinality of $s_i$, we see that there are only $3^{r_1}4^{r_2}$ possibilities for the word $Y$ (corresponding to a maximal triple). On the other hand, the set $s_2$ is determined by $r_2$ as it consists of the last $r_2$ elements, and $s_1$ is a subset of size $r_1$ of the first $n - r_2 = n - r + r_1$ elements.

Hence, using \eqref{triple-count} once more and recalling that $\theta_{2}=\sqrt{3}/2$, we see that
$$W(n, \delta)\leq \left(\sum\limits_{r_1 = 0}^r \binom{n - r + r_1}{r_1}3^{r_1}4^{r - r_1}\right)\left(\frac{\sqrt{3}}{2}\right)^{n-r} 4^{n - r}.$$

We use $\binom{n - r + r_1}{r_1}\leq \binom{n}{r_1}$ and the Chernoff bound for the tail of the binomial distribution to get
\begin{align*}
  \sum\limits_{r_1 = 0}^r \binom{n}{r_1}3^{r_1}4^{n - r_1} & \leq 7^n\exp\left\{-n\left(\delta\log\left(\frac{\delta}{3/7}\right) +
  (1 - \delta)\log\left(\frac{1-\delta}{4/7} \right)\right) \right\}                                                                \\
                                                           & = \exp\{n[h(\delta)+\delta \log 3 + (1-\delta)\log 4]\}.
\end{align*}
Therefore, we get the improved bound
\begin{align*}
  \P(\ell(X^{(n)}) < n\delta) & \leq |W(n,\delta)|4^{-n}                                                 \\
                              & \le  \exp\{n[h(\delta)+\delta \log(3/4)+(1-\delta) \log(\sqrt{3}/2)] \}.
\end{align*}
However the improved lower bound only gives a marginal improvement to $0.034$, i.e., at least 3.4\% of the letters are unmatched on average.

\section{An elementary upper bound on \texorpdfstring{$\lambda_2$}{L2}}\label{S:upperbound}
We claim that $\lambda_2\le 0.29$ (we shall use more sophisticated methods to obtain a better bound in Section~\ref{Sec:greedy}).  This is achieved as follows. Let $U_{1},V_{1},U_{2},V_{2},\ldots $ be i.i.d. Geometric($1/2$) random variables, i.e.,  $\P[U_{1}=j]=2^{-j}$ for $j\ge 1$. Then $\E[U_{1}]=2$, and hence with $m=\lfloor n/4\rfloor$ we get $N_{n}:=U_{1}+V_{1}+\ldots +U_{m}+V_{m}=n+O(\sqrt{n})$ with high probability. We create a string $S\in \{\alpha,\bar{\alpha},\beta,\bar{\beta}\}^{N}$ by setting down a random string of $\{\alp,\bar{\alp}\}$ of length $U_{1}$, then a random string of $\{\beta,\bar{\beta}\}$ of length $V_{1}$, etc. Thus, $U_{i},V_{i}$ are the length of runs of the two species of symbols. This makes the length of the string random but since it is in a $\sqrt{n}$ length window of $n$, this should not change anything much (as regards the proportion of unpaired sites). Consider the following matching algorithm.

Fix any maximal noncrossing matching of all the $\beta,\bar{\beta}$ symbols.  Then we make the best possible non-crossing matching of each run of $\alpha,\bar{\alpha}$ within itself.  Thus, if the first run happens to be $\alpha,\bar{\alpha},\alpha$, then, we could match up the first two sites and leave the third one unpaired.

In this matching scheme, in the first stage there are $O(\sqrt{n})$ unpaired sites (the difference between the number of $\beta$ and the number of $\bar{\beta}$ symbols in $S$). For the second stage, note that in the $j$th run (the one that has length $U_{j}$), the number of $\alpha$-symbols is $\xi_{j}\sim \mb{Binomial}(U_{j},\frac{1}{2})$, and hence the number of left overs is $|2\xi_{j}-U_{j}|$. The total number of left over sites has expectation $m\E[|2\xi_{1}-U_{1}|] + O(\sqrt{n})$ which gives us the bound
\begin{align*}
  \lambda_2\le \frac{1}{4}\E[|2\xi_{1}-U_{1}|].
\end{align*}
Numerical evaluation of the expectation (expressed as an infinite sum) gives the bound $\lambda_2\le 0.2886...$.

\section{Concentration around expected behaviour}\label{sec:concentration}
We prove Proposition~\ref{prop:concentration} in this section.
The tool is the well-known Hoeffding's inequality for sums of martingale differences (see section 4.1 of Ledoux~\cite{ledoux} for a proof and the book of Steele~\cite{steele} for its use in many combinatorial optimization problems similar to ours). It says that if  $d_{1},d_{2},\ldots ,d_{n}$ is a martingale difference sequence, that is $\E[d_{j}\vert d_{1},\ldots d_{j-1}]=0$ for each $j$ (for $j=1$ this is to be interpreted as $\E[d_{1}]=0$) and $|d_{j}|\le B_{j}$ with probability $1$ for some constant $B_{j}$, then for any $t>0$, we have
\begin{equation} \label{eq:hoeffding}
  \P\l[\big| \sum_{j=1}^{n}d_{j} \big| >t \r]\le 2e^{-\frac{t^{2}}{2(B_{1}^{2}+\ldots +B_{n}^{2})}}.
\end{equation}

\begin{proof}[Proof of Proposition~\ref{prop:concentration}] Let $X=X^{(n)}=(X_{1},\ldots ,X_{n})$. Define for $j=1,\ldots n$,
  $$
    d_{j} = \E\l[L(X_{1},\ldots ,X_{n}) \given X_{1},\ldots ,X_{j} \r] - \E\l[L(X_{1},\ldots ,X_{n}) \given X_{1},\ldots ,X_{j-1} \r].
  $$
  Then $d_{j}$ is a martingale difference sequence by the tower property
  \begin{align*}
    \E[\E[U \given V,W]\given W]=\E[U\given W].
  \end{align*}
  Further, $L(X)-\E[L(X)] = d_{1}+\ldots +d_{n}$. If we show that $|d_{j}|\le 2$, then by applying Hoeffding's inequality~\eqref{eq:hoeffding}, we get the statement in the lemma.

  To prove that $|d_{j}|\le 2$, fix $j$ and let $Y=(X_{1},\ldots ,X_{j-1},X_{j}',X_{j+1},\ldots ,X_{n})$ where $X_{j}'$ is an independent copy of $X_{j}$ that is also independent of all $X_{j}$s. Then,
  $$
    \E[L(Y)\given X_{1},\ldots ,X_{j}] \; =\;  \E[L(Y)\given X_{1},\ldots ,X_{j-1}]  \; =  \; \E[L(X)\given X_{1},\ldots ,X_{j-1}],
  $$
  where the first equality holds because $X_{j}$ is independent of $Y$ and the second equality holds because $X_{1},\ldots X_{j-1}$ bear the same relationship to $X$ as to $Y$. Thus, we conclude that
  $$
    d_{j} = \E\l[L(X)-L(Y) \given X_{1},\ldots X_{j} \r].
  $$
  But $X$ and $Y$ differ only in one co-ordinate. From any non-crossing matching of $X$, by deleting the edge (if any) matching the $j$th co-ordinate, we obtain a non-crossing matching for $Y$ with at most two more unmatched indices. Therefore $L(Y)\le L(X)+2$ and by symmetry between $X$ and $Y$, we get $|L(X)-L(Y)|\le 2$. Therefore $|d_{j}| \le \E\l[|L(X)-L(Y)| \given X_{1},\ldots X_{j} \r] \le 2$.
\end{proof}

\section{Greedy algorithms and Upper bounds}\label{Sec:greedy}
%

The goal of this section is to prove Proposition~\ref{prop:onesidedgreedy}. First we introduce a Markov chain related to this algorithm. Recall the description of the algorithm from the introduction.

\subsection{The associated Markov chain and its stationary distribution} For simplicity of notation, we write the alphabet set as $\mathbb A=\{1,\bar{1},\ldots ,k,\bar{k}\}$. Let $w[t]$ be the word formed by all the {\em accessible} letters at ``time'' $t$ -- these are the letters among $X_1,\ldots ,X_t$ that are   still available for matching in future in the above greedy algorithm. Then $w[t]$ is a Markov chain whose state space is  $\Omega=\mathbb A^0\sqcup \mathbb A^1\sqcup \mathbb A^2\sqcup \ldots$, the set of all finite strings in the  alphabet $\mathbb A$ (including the empty string) and whose dynamics  are as follows:

If $w[t]=(w_{1},\ldots ,w_{p})$ and $X_{t+1}=x$, then $w[t+1]=(w_1,\ldots ,w_p,x)$ if $\bar{x}$ does not occur in $w[t]$. Otherwise  $w[t+1]=(w_{1},\ldots ,w_{j-1})$ where $j$ is the largest index such that $w_j=\bar{x}$. Two letters get matched each time the length of $w[t]$ reduces. Hence the number left unmatched after $n$ steps is $n-2\sum_{t=2}^n\mathbf 1_{\mb{\tiny length}(w[t])<\mb{\tiny length}(w[t-1])}$.

The Markov chain  is not irreducible. From any state it is possible to go to $\emptyset$ but from $\emptyset$ the chain can only go to states in
\[
  \Omega_0=\{w\in \Omega \; : \; \mb{ at most one of }x,\bar{x}\in w\mb{ for each }x\in \mathbb A\}
\]
which makes $\Omega_0$ the unique irreducible class. As we shall show next, this Markov chain has a stationary probability distribution $\pi$. By the general theory of Markov chains, the stationary distribution is unique. To give the formula for $\pi$, we need some notation.

For a word $w\in \Omega_0$, define $a_i(w)$ inductively by declaring $a_1(w)+\ldots +a_j(w)$ to be the length of the maximal initial segment in $w$ (reading from the left) containing at most $j$ distinct symbols. Note that if $w$ has only $j$ different symbols from $\mathbb A$, it follows that $a_i(w)=0$ for $i\ge j$. In particular, as $w\in \Omega_0$, it has at most $k$ distinct symbols.  For example, if $k=3$ and $w=11212212311232$, then $(a_1,a_2,a_3)=(2,8,6)$. If $w=22212$ then $(a_1,a_2,a_3)=(3,2,0)$. For the empty word, $a_i(w)=0$ for all $i$.

\begin{proposition}\label{prop:statdist} Fix $k\ge 2$. Let $\tau_j=\frac{j+1}{j(2k+1)-1}$ for $1\le j\le k$. Then the unique stationary probability distribution is given by
  \begin{align*}
    \pi(w)=\frac{1}{\mb{Z}}\tau_1^{a_1(w)}\tau_2^{a_2(w)}\ldots \tau_k^{a_k(w)}
  \end{align*}
  where $Z=\sum_{r=0}^k\binom{k}{r}2^r\prod_{j=1}^r\frac{j\tau_j}{1-j\tau_j}=\sum_{r=0}^k\binom{k}{r}2^r\prod_{j=1}^r\frac{j(j+1)}{j(2k-j)-1}$.
\end{proposition}
Assuming this proposition, we prove Proposition~\ref{prop:onesidedgreedy}.

\subsection{Proof of Proposition~\ref{prop:onesidedgreedy}}  From the earlier observation, the expected proportion of matched letters after $n$ steps is
\begin{align*}
  \frac{2}{n}\sum_{k=1}^n\P\{\mb{length}(w[t])<\mb{length}(w[t-1])\} \to 2\P_{\pi}\{\mb{length}(w[1])<\mb{length}(w[0])\}
\end{align*}
where the subscript $\pi$ is to indicate that $w[0]$ is sampled from $\pi$ (in the actual chain, we start with $w[0]=\emptyset$) and the convergence follows from the general theory of Markov chains which asserts that the distribution of $(w[t-1],w[t])$ (from any starting point) converges to the distribution of $(w[0],w[1])$ when $w[0]$ has distribution $\pi$. As a consequence, we arrive at the upper bound
\begin{align}\label{eq:boundforlambdakintermsofstatdist}
  \lambda_k\le 1-2\P_{\pi}\{\mb{length}(w[1])<\mb{length}(w[0])\}.
\end{align}
If $w$ has $r$ distinct symbols, then $a_r(w)>0 (=a_{r+1}(w))$ and its length gets reduced if and only if the next arriving letter can match up with one of them, i.e., with probability $\frac{r}{2k}$.  Further, for a given choice of strictly positive integers $a_1,\ldots ,a_r$, the number of words $w$ with $a_i(w)=a_i$ is precisely
\begin{align}\label{eq:numberofwforgivenai}
  2^rk(k-1)\ldots (k-r+1)2^{a_2-1}3^{a_3-1}\ldots r^{a_r-1}.
\end{align}
Here $2k-2i+2$ is for the choice of $i$th new symbol (the locations are determined by $a_1,\ldots ,a_r$) and the $a_j-1$ letters between the $j$th new symbol and $(j+1)$st new symbol each have $j$ choices, hence the factor of $j^{a_j-1}$.
Thus,
\begin{align*}
  \P_{\pi}\{\mb{length}(w[1])<\mb{length}(w[0])\} & =\frac{1}{2k Z}\sum_{r=1}^kr2^r\binom{k}{r}\sum_{a_i\ge 1: i\le r}\prod_{j=1}^r(j \tau_j)^{a_j} \\
                                                  & =\frac{1}{2k Z}\sum_{r=1}^kr2^r\binom{k}{r}\prod_{j=1}^r\frac{j\tau_j}{1-j\tau_j}.
\end{align*}
Substituting the value of $\tau_j$ given in the statement of the proposition,  \begin{align}\label{eq:boundforlambdak}
  \tilde{\lambda}_k= 1-\frac{1}{kZ}\sum_{r=1}^kr2^r\binom{k}{r}\prod_{j=1}^r\frac{j(j+1)}{j(2k-j)-1}.
\end{align}
Plugging in the expression for $Z$ given in Proposition~\ref{prop:statdist} completes the proof of  Proposition~\ref{prop:onesidedgreedy}. \hfill $\qed$

\noindent{\bf Case $k=2$:} This is the case we care most about. We see that $\tau_1=\frac12$ and $\tau_2=\frac13$ and $Z=13$. Hence $\pi(w)=\frac{1}{13}\l(\frac32\r)^{a_1(w)}\l(\frac13\r)^{\mb{ \tiny length}(w)}$ where $a_1(w)$ is the length of the first run (i.e., the maximum $j$ such that $w_1=w_2=\ldots =w_j$). Therefore, \eqref{eq:boundforlambdak} becomes
\[
  \tilde{\lambda}_2= 1-\frac{1}{2\times 13}\l( 4+16\r)=\frac{3}{13}=0.2307\ldots
\]

\subsection{Proof of Proposition~\ref{prop:statdist}} Let $\sigma(w)=\tau_1^{a_1(w)}\ldots \tau_k^{a_k(w)}$. If $w$ has $r$ distinct symbols, then $a_j\ge 1$ for $j\le r$ and $a_j=0$ for $j>r$. The number of words $w$ with given $a_1,\ldots ,a_r$ is given in \eqref{eq:numberofwforgivenai}. Hence the sum of $\sigma(w)$ over such $w$ is
\[
  2^r\binom{k}{r}\sum_{a_1,\ldots ,a_r\ge 1}\prod_{j=1}^r(j\tau_j)^{a_j}=2^r\binom{k}{r}\prod_{j=1}^r\frac{j\tau_j}{1-j\tau_j}.
\]
Sum over $r$ (including $r=0$) to get the given expression for $Z$.

It suffices to check that $\sigma$ satisfies the equations for the stationary distribution, since we know the uniqueness (up to scalar multiples) of stationary distribution. The general equations are
\[
  \sum_{w:w'\mapsto w}\sigma(w')=2k\sigma(w)
\]
where the notation $w'\mapsto w$ means that $w'$ can lead to $w$ in one step (in our Markov chain, a given $w'$ can lead to a given $w$ in at most one way, hence the transition probability is exactly $1/2k$). If $w=(w_1,\ldots ,w_p)$ has  exactly $r$ distinct symbols, then $a_r(w)>0=a_{r+1}(w)$, and   $\sigma(w)=\tau_1^{a_1(w)}\ldots \tau_r^{a_r(w)}$. The possible $w'$ are:
\begin{enumerate}
  \item $w'=(w_1,\ldots ,w_{p-1})$. Then $a_i(w')=a_i(w)$ for $i\le r-1$ and $a_r(w')=a_r(w)-1$.
  \item $w'=wxy^1t_1y^2\ldots t_jy^{j+1}$ where $x\in \mathbb A$ is a symbol that occurs in $w$ and $t_i\in \mathbb A$ are the new symbols that did not occur before and $y^i=(y_{1}^i,\ldots,y^i_{m_i})$ with $m_i\ge 0$ . Here $j$ can vary from $0$ to $k-r$. Further, $x$ should not occur in $y^1t_1y^2\ldots t_jy^{j+1}$ so that $w'$ can lead to $w$ when an $\bar{x}$ arrives (it is tacit that all our words are in $\Omega_0$, so we do not write  those conditions again). Then
        \[
          a_{i}(w')=\begin{cases} a_i(w) & \mb{ if }i\le r-1, \\ a_r(w)+m_1+1 &\mb{ if }i=r, \\ m_{i-r+1}+1 &\mb{ if }r+1\le i\le r+j. \end{cases}
        \]
        For given $j$ and $m_1,\ldots ,m_j$, the number of choices of such $w'$ is
        \[
          2^j(k-r)(k-r-1)\ldots (k-r-j+1) \times r (r-1)^{m_1}r^{m_2}\ldots (r+j-1)^{m_{j+1}}.
        \]
        This is because there are $2k-2r-2i+2$ choices for $t_i$ and $r+i-2$ choices for each letter in $y^i$.

  \item $w'=wt_1y^1t_2y^2\ldots t_jy^{j}$ where  $y^i=(y_{1}^i,\ldots,y^i_{m_i})$ with $m_i\ge 0$ and $t_i\in \mathbb A$ are the new symbols that did not occur before. Here $j$ can vary from $1$ to $k-r$. Further, $t_1$ should not occur in $y^1t_2y^2\ldots t_jy^{j}$. Then
        \[
          a_{i}(w')=\begin{cases} a_i(w) & \mb{ if }i\le r,  \\ m_{i-r}+1 &\mb{ if }r+1\le i\le r+j. \end{cases}
        \]
        For given $j$ and $m_1,\ldots ,m_j$, the number of choices of such $w'$ is
        \[
          2^j(k-r)(k-r-1)\ldots (k-r-j+1) \times r^{m_1}(r+1)^{m_2}\ldots (r+j-1)^{m_{j}}.
        \]
        Here $2k-2r-2i+2$ is the number of choices for $t_i$ and $r+i-1$ is the number of choices for each letter in $y^i$.
\end{enumerate}
Using these and cancelling common factors, the equation for stationary distribution becomes
\begin{align*}
  2k\tau_r & =1+\frac{r\tau_r^2}{1-(r-1)\tau_r}\sum_{j=0}^{k-r} 2^j(k-r)_{j\downarrow} \prod_{i=r+1}^{r+j}\frac{\tau_i}{1-(i-1)\tau_i}                                              \\
           & \;\;\;\;\;\;\;+\tau_r\sum_{j=1}^{k-r} 2^j(k-r)_{j\downarrow} \prod_{i=r+1}^{r+j}\frac{1}{1-(i-1)\tau_i}                                                                \\
           & =1+\frac{r\tau_r^2}{1-(r-1)\tau_r}+\l(\frac{r\tau_r^2}{1-(r-1)\tau_r}+\tau_r\r)\sum_{j=1}^{k-r} 2^j(k-r)_{j\downarrow} \prod_{i=r+1}^{r+j}\frac{\tau_i}{1-(i-1)\tau_i}
\end{align*}
which is the same as (empty products are interpreted as $1$)
\[
  (2k+1)\tau_r-1=\frac{\tau_r(\tau_r+1)}{1-(r-1)\tau_r}\sum_{j=0}^{k-r}  \prod_{i=r+1}^{r+j}\frac{2(k+1-i)\tau_i}{1-(i-1)\tau_i}, \;\;\; 0\le r\le k.
\]
Notice that the sum on the right is of the form
$1+u_{r+1}+u_{r+1}u_{r+2}+\ldots = 1+u_{r+1}\l( 1+u_{r+2}+u_{r+2}u_{r+3}+\ldots \r)$, and the quantity in brackets on the right occurs in exactly that form in the equation for $r+1$. Therefore, the above equation can be re-written for $r<k$ as
\[
  \frac{((2k+1)\tau_r-1)(1-(r-1)\tau_r)}{\tau_r(\tau_r+1)}=1+\frac{2(k-r)\tau_{r+1}}{1-r\tau_{r+1}}\frac{((2k+1)\tau_{r+1}-1)(1-r\tau_{r+1})}{\tau_{r+1}(\tau_{r+1}+1)}.
\]
Plugging in the stated values of $\tau_r$ and $\tau_{r+1}$, a short calculation shows that both sides are equal to $(2k+2-r)/r$, hence equality holds.

For $r=k$, the original equation is $(2k+1)\tau_k-1=\frac{\tau_k(\tau_k+1)}{1-(k-1)\tau_k}$ which is easily seen to be satisfied by $\tau_k=1/(2k-1)$.  This completes the proof. \hfill \qed

\begin{remark} Although the proof is more or less straightforward checking with some calculations, it hinged on having the form of the stationary distribution. All features of the stationary distribution, namely the product form with exponents being $a_i$s and the values of $\tau_i$s were arrived at by extensive checking on Mathematica software for several values of $k$, along with some guess work. On a computer, one must restrict to finite state space chains, and a natural restriction is to words of length at most $L$ (steps outside this are forbidden). If $\pi_L$ is the stationary distribution of this Markov chain, then not only does $\pi_L$ converge to $\pi$, but curiously $\pi_L(w)=\pi(w)$ for all $w$ of length $L-1$ or less!

\end{remark}

\bibliographystyle{amsplain}

\end{document}